\documentclass[12pt,twoside]{amsart}
\usepackage{amsmath}
\usepackage{amssymb}
\usepackage{graphicx}
\usepackage{xypic}
\xyoption{all}
\usepackage{multirow}
\usepackage[
singlelinecheck=false 
]{caption}
\title[Configuration spaces of commuting elements]{Configuration spaces of commuting elements}

\author{Jos\'e Cantarero}
\address{
\hfill\break Consejo Nacional de Ciencia y Tecnolog\'ia \\
\hfill\break Centro de Investigaci\'on en Matem\'aticas, A.C. Unidad M\'erida \\
\hfill\break Parque Cient\'ifico y Tecnol\'ogico de Yucat\'an \\ 
\hfill\break Carretera Sierra Papacal--Chuburn\'a Puerto Km 5.5 \\
\hfill\break Sierra Papacal, M\'erida, YUC 97302 \\
\hfill\break Mexico.}
\email{cantarero@cimat.mx}

\author{\'Angel R. Jim\'enez}
\address{
\hfill\break Universidad Aut\'onoma de Chiapas \\
\hfill\break Facultad de Ciencias en F\'isica y Matem\'aticas \\
\hfill\break Carretera Emiliano Zapata Km 8 \\
\hfill\break Rancho San Francisco, Ciudad Universitaria, Ter\'an \\
\hfill\break Tuxtla Guti\'errez, Chiapas 29050 \\
\hfill\break Mexico.}
\email{rolando.jimenez@unach.mx}

\newcommand{\ab}{\operatorname{ab}\nolimits}
\newcommand{\Conf}{\operatorname{Conf}\nolimits}
\newcommand{\Der}{\operatorname{Der}\nolimits}
\newcommand{\Hom}{\operatorname{Hom}\nolimits}
\newcommand{\mult}{\operatorname{mult}\nolimits}
\newcommand{\op}{\operatorname{op}\nolimits}
\newcommand{\pr}{\operatorname{pr}\nolimits}
\newcommand{\sgn}{\operatorname{sgn}\nolimits}
\newcommand{\Spin}{\operatorname{Spin}\nolimits}
\newcommand{\std}{\operatorname{std}\nolimits}
\newcommand{\Top}{\operatorname{Top}\nolimits}
\newcommand{\UConf}{\operatorname{UConf}\nolimits}

\def \P{{\mathcal P}}
\def \Q{{\mathbb Q}}
\def \R{{\mathbb R}}
\def \Z{{\mathbb Z}}

\theoremstyle{plain}
\newtheorem*{introtheorem}{Theorem}
\newtheorem{theorem}{Theorem}[section]
\newtheorem{proposition}[theorem]{Proposition}
\newtheorem{corollary}[theorem]{Corollary}
\newtheorem{lemma}[theorem]{Lemma}

\theoremstyle{definition}
\newtheorem{definition}[theorem]{Definition}
\newtheorem{remark}[theorem]{Remark}
\newtheorem{example}[theorem]{Example}



\begin{document}

\begin{abstract}
In this article we introduce the space of configurations of commuting
elements in a topological group and show that it satisfies rational 
homological stability for the sequences of unitary, special unitary and symplectic 
groups. We also prove that it satisfies cohomological rational representation stability 
with respect to the number of elements in the tuple for finite products of
such groups, in particular cohomological rational stability for the space
of unordered configurations of commuting elements. Finally we present 
some computations of cohomology in the unstable range.  
\end{abstract}

\maketitle

\section*{Introduction}
\label{SecIntroduction}

The study of the space $\Hom(\Z^k,G)$ of commuting $k$-tuples in a compact Lie group
$G$ was first motivated by computations in theoretical physics. In 
\cite{Wit}, Witten computed the supersymmetric index of vacuum states
in Yang-Mills theory on a spatial three-dimensional torus. This computation
did not agree with the expected value and this disagreement was found
in \cite{Wit2} to come from the assumption that the space of conjugacy
classes of commuting $3$-tuples in a connected compact Lie group $G$ is path-connected, 
which does not hold for instance when $G=\Spin_7$. Further computations
of this index were performed in \cite{KS}.

Goldman analyzed these spaces in \cite{Go} from a geometric point of view
and Adem-Cohen revisited them in \cite{AC} focusing on topological and
homotopical properties. This influential paper led to a series of articles
which examined the spaces of commuting $k$-tuples and generalizations from
the perspective of algebraic topology, and computed their homotopy
invariants such as the fundamental group \cite{GPS}, rational cohomology \cite{B},
equivariant $K$-theory \cite{AG} or twisted equivariant $K$-theory \cite{ACG}.
Another interesting feature of these spaces is that they can be used to
define new cohomology theories, such as commutative $K$-theory \cite{AG2}
and nilpotent $K$-theory \cite{AGLT}. The survey \cite{CS} gives a nice
overview of recent developments in this area.

On the other hand, configuration spaces and their natural variations 
(labelled, fibrewise, orbit configuration spaces) play an important
role in geometry and topology. See for instance \cite{Co} and \cite{FH}.
Hence it is natural to consider a hybrid of these two concepts, the 
space of configurations of commuting elements in a topological group. 
The study of this space was suggested in \cite{RS}, and was first 
developed in \cite{Ji}.

In this article we examine the homology and cohomology of the space
$\Conf_k^{\ab}(G)$ of configurations of $k$ commuting elements in a
compact Lie group $G$, including and extending results from \cite{Ji}.
When $G$ belongs to the family of $\P$ of finite products of classical groups of the form $U_n$, $SU_n$ 
and $Sp_n$, and $T$ is a maximal torus of $G$, the conjugation action of $G$ on $\Conf_k(T)$ determines
a cohomological principal bundle $G \times \Conf_k(T) \to \Conf_k^{\ab}(G)$
in the sense of \cite{B} for cohomology over fields whose characteristics do
not divide the order of the Weyl group of $G$. This is the main tool used in
our results.

Comparing the homology of $\Conf_k(T)$ and $T^k$, we find that the inclusion 
of $\Conf_k^{\ab}(G)$ into $\Hom(\Z^k,G)$ induces an isomorphism on $n$-dimensional
homology over these fields if $n$ is sufficiently large. Stability results in \cite{RS} 
and \cite{KT} imply the following result, which is Theorem \ref{StabilityForRank} in the text.

\begin{introtheorem}
Let $G_r$ be one the classical groups $U_r$, $SU_{r+1}$ or $Sp_r$. For each $k \geq 1$, the sequence $\{ \Conf_k^{\ab}(G_r) \}_{r \geq 1}$ satisfies strong rational homological stability. Stability holds in homological degree $n$ once
\[ r \geq \left\{ \begin{array}{ll}
                  n+2 & \text{ if } G_r = Sp_r, \\
                  \max \{ (n+k-1)/2 , n+2 \} & \text{ if } G_r = U_r, \\
                  \max \{ (n+k-3)/2 , n+2 \} & \text{ if } G_r = SU_{r+1}. \end{array} \right. \]
\end{introtheorem}

Configuration spaces and spaces of commuting $k$-tuples in a compact Lie group 
do not satisfy homological stability with respect to $k$. However, they both 
satisfy a certain representation stability phenomenon in the sense of \cite{CF}.
If $G$ is a connected, compact Lie group, Ramras-Stafa \cite{RS} show that the 
unit components of the spaces $\Hom(\Z^k,G)$ satisfy uniform representation 
stability in homology over fields whose characteristics do not divide the
order of the Weyl group of $G$. And Church \cite{Ch} proved that if $M$ is a 
connected, compact orientable manifold, the spaces $\Conf_k(M)$ satisfy 
uniform representation stability in rational cohomology. We prove a similar
behaviour for the spaces $\Conf_k^{\ab}(G)$ with $G$ in $\P$.

\begin{introtheorem}
Let $G$ be a group in $\P$ and $j$ a nonnegative integer. Then the
sequences of $\Sigma_k$--representations $\{ H^j(\Conf_k^{\ab}(G);\Q) \}_{k \geq 1}$
satisfy uniform representation stability.
\end{introtheorem}

As a consequence we obtain rational cohomological stability for the space \newline
\mbox{$\UConf_k^{\ab}(G) = \Conf_k^{\ab}(G)/\Sigma_k$} of unordered configurations of $k$ commuting elements,
although not necessarily induced by maps $\UConf_k^{\ab}(G) \to \UConf_{k-1}^{\ab}(G)$.  

The cohomological principal bundle $G \times \Conf_k(T) \to \Conf_k^{\ab}(G)$ is
a useful tool to compute other homotopy invariants, even outside of the stable range. 
For instance, we use it to determine the homotopy type of $\Conf_k^{\ab}(SU_2)$ for all $k \geq 2$.
We also compute other cohomology groups over fields whose characteristics do not divide the
order of the Weyl group in certain cases. Namely, the cohomology groups of $\Conf_3^{\ab}(U_2)$, 
of $\Conf_2^{\ab}(G)$ for any group $G$ of rank two in the family $\P$, and the first cohomology 
group of $\Conf_k^{\ab}(G)$ for any group $G$ of rank at least two
in the family $\P$. We also determine the cohomology rings of $\Conf_2^{\ab}(G)$ and $\UConf_2^{\ab}(G)$
for $G=U_2$, $S^1 \times SU_2$ over any field of characteristic different from two.

\section{Configurations of commuting elements}
\label{SecConfigurations}

In this section we introduce the space of configurations of commuting 
elements in a topological group. For compact Lie groups which are 
products of unitary, special unitary and symplectic groups, we give
a description of its homology groups and cohomology rings over certain 
fields in terms of their maximal tori and Weyl groups.

\begin{definition}
Let $G$ be a topological group and $k$ a positive integer. The space of 
configurations of $k$ commuting elements in $G$ is given by
\[ \Conf_k^{\ab}(G) = \{ (g_1,\ldots,g_k) \in \Conf_k(G) \mid g_ig_j=g_jg_i \text{ for all } i, j \} \]
with the subspace topology from $G^k$.
\end{definition}

Considering the usual identification between $\Hom(\Z^k,G)$ and the space
of commu-ting $k$-tuples, this space is just $\Hom(\Z^k,G) \cap \Conf_k(G)$. 
Note that $\Conf_1^{\ab}(G)=G$, so we will often ignore this case.

\begin{example}
We can identify the homotopy type of $\Conf_2^{\ab}(SU_2)$ by elementary means.
Consider the projection $\pr_1 \colon \Conf_2^{\ab}(SU_2) \to SU_2$ to the
first coordinate and the map 
\begin{gather*}
g \colon SU_2 \to \Conf_2^{\ab}(SU_2) \\
 A \mapsto (A,-A).
\end{gather*}
It is clear that $\pr_1g$ is the identity map, while $g \pr_1$ is homotopic to the
identity via the following homotopy
\begin{gather*}
H \colon \Conf_2^{\ab}(SU_2) \times I \to \Conf_2^{\ab}(SU_2) \\
H(A,B,t) = \left( A , \frac{t(-A)+(1-t)B}{\|t(-A)+(1-t)B\|} \right).
\end{gather*}
where we are regarding the matrices in $SU_2$ as norm one elements in $\R^4$.
Since $A \neq B$, the second coordinate of $H(A,B,t)$ makes sense and it is different from
$A$. If we regard now $A$ and $B$ as commuting quaternions, then their imaginary
parts must be parallel. Since the imaginary part of the second coordinate is a linear
combination of the imaginary parts of $A$ and $B$, it is also parallel to the imaginary
part of $A$. Therefore $H$ is well defined and $\Conf_2^{\ab}(SU_2)$ is homotopy equivalent to $SU_2$. 
\end{example}

In order to identify the homotopy type of $\Conf_k^{\ab}(SU_2)$ for $k \geq 3$, we need to establish 
some properties of the conjugation action of $G$ on $\Conf_k^{\ab}(G)$ for certain compact Lie groups. 
Let $T$ be a maximal torus of $G$ and consider the restriction $ \phi \colon G \times \Conf_k(T) \to \Conf_k^{\ab}(G)$ of the action map.
Following \cite{AG}, we let $\P$ be the family of compact Lie groups which are
products of groups of the form $SU_n$, $U_n$ and $Sp_n$.

\begin{proposition}
\label{CohomologyFormula}
Let $G$ be a group in the family $\P$ with maximal torus $T$ and Weyl group $W_G$.
If $F$ is a field whose characteristic does not divide the order of $W_G$, then $(\phi,N_G(T))$ 
is a cohomological principal bundle for $H^*(-;F)$ and the map $\phi$ induces an isomorphisms
of graded algebras
\[ H^*(\Conf_k^{\ab}(G);F) \cong H^*(G/T \times \Conf_k(T);F)^{W_G} \]
and an isomorphism $ H_n(\Conf_k^{\ab}(G);F) \cong H_n(G/T \times \Conf_k(T);F)^{W_G}$
for each $n \geq 0$.
\end{proposition} 

\begin{proof}
The space $\Conf_k^{\ab}(G)$ is a closed subspace of $G^k$, hence it is
paracompact Hausdorff. The conjugation action of $G$ on $\Conf_k^{\ab}(G)$ 
is the restriction of the conjugation action on $\Hom(\Z^k,G)$. Since $G$
belongs to the family $\P$, its fundamental group is torsion-free, hence
$\Hom(\Z^k,G)^T =T^k$ by Subsection 6.1 and the paragraph before Theorem 6.7
in \cite{AG}. Then we have 
\[ \Conf_k^{\ab}(G)^T = \Hom(\Z^k,G)^T \cap \Conf_k^{\ab}(G) = T^k \cap \Conf_k^{\ab}(G) = \Conf_k(T) . \]
By Theorem 3.3 in \cite{B}, the pair $(\phi,N_G(T))$ is a cohomological 
principal bundle for $H^*(-;F)$ for any field $F$ whose characteristic
does not divide the order of $W_G$ and the cohomological formula follows. 

Note that the second part of the proof of Proposition 3.2 in \cite{RS} shows
that if $X$ is a space on which a finite group $H$ acts, and the quotient map
$q \colon X \to X/H$ satisfies that $q^* \colon H^n(X/H;F) \to H^n(X;F)^H$ 
is an isomorphism, then the restriction $H_n(X;F)^H \to H_n(X/H;F)$ of $q_*$ 
is an isomorphism. Therefore we obtain the statement for homology.
\end{proof}

\begin{remark}
The isomorphisms in this proposition are natural with respect to monomorphisms $f \colon G \to G'$
which satisfy $ f(T) \leq T'$ and $f(N_G(T)) \leq N_{G'}(T')$.
\end{remark}

Any group $G$ in the family $\P$ is path-connected. If the dimension of its maximal torus $T$ is 
at least two, then $\Conf_k(T)$ is path-connected by Theorem 3.2 in \cite{GGG} (see also
Theorem 1 in \cite{Bir}). The next result follows.

\begin{corollary}
If $G$ is a group in the family $\P$ of rank at least two, then $\Conf_k^{\ab}(G)$ is path-connected.
\end{corollary}

We now determine the homotopy type of $\Conf_k^{\ab}(SU_2)$ for $k \geq 3$. Let $T$ be the standard 
maximal torus of $SU_2$. Since the map $\phi \colon SU_2/T \times \Conf_k(T) \to \Conf_k^{\ab}(SU_2)$ is a 
cohomological principal bundle, it is a closed surjection and therefore so is the induced map 
$ \Phi \colon SU_2/T \times_{\Z/2} \Conf_k(T) \to \Conf_k^{\ab}(SU_2)$.

In this case the map $\Phi$ is injective as well. If $\Phi([gT,z_1,\ldots,z_k]) = \Phi([hT,w_1,\ldots,w_k])$,
then $ h^{-1}gz_ig^{-1}h=w_i$ for all $i$. Since $k \geq 3$, there exists
a $z_j$ which is not in the center of $SU_2$, from where $ h^{-1}g \in N_{SU_2}(T)$
and so $[gT,z_1,\ldots,z_k]) = [hT,w_1,\ldots,w_k]$. Therefore $\Phi$ is a homeomorphism.

The space $ SU_2/T$ is $\Z/2$--homeomorphic to $S^2$ with the antipodal action.
On the other hand, $S^1 \times \Conf_{k-1}(0,1) \cong \Conf_k(T) $ via the
homeomorphism that sends $(z,t_1,\ldots,t_{k-1})$ to $(z,ze^{2\pi it_1},\ldots,ze^{2\pi it_{k-1}})$
by Example 2.6 in \cite{Co}. The space $\Conf_{k-1}(0,1)$ is homotopically discrete, its path components
given by the subspaces
\[ A_{\sigma} = \{ (t_1,\ldots,t_{k-1}) \in (0,1)^{k-1} \mid t_{\sigma^{-1}(1)} < \ldots < t_{\sigma^{-1}(k-1)} \} \]
for each $\sigma \in \Sigma_{k-1}$. These spaces are convex, hence contractible. 
The action of $\Z/2$ on $\Conf_k(T)$ corresponds
to the action
\[ (z,t_1,\ldots,t_{k-1}) \mapsto (z^{-1},1-t_1,\ldots,1-t_{k-1}) \]
on $S^1 \times \Conf_{k-1}(0,1)$. Note that if $(t_1,\ldots,t_{k-1}) \in A_{\sigma}$,
then $(1-t_1,\ldots,1-t_{k-1}) \in A_{\overline{\sigma}}$, where $\overline{\sigma}(i)=k-\sigma(i)$.
Therefore $\Conf_k(S^1)$ is $\Z/2$--homotopy equivalent to $ S^1 \times \Sigma_{k-1}$, where
we consider $\Sigma_{k-1}$ as a discrete space, with the action of $\Z/2$ given by
$(z,\sigma) \mapsto (z^{-1},\overline{\sigma})$. Therefore if $k \geq 3$, we have
\[ \Conf_k^{\ab}(SU_2) \simeq S^2 \times_{\Z/2} (S^1 \times \Sigma_{k-1}) \cong \coprod_{i=1}^{(k-1)!/2} S^2 \times S^1, \]
where the last homeomorphism holds since for each $\sigma \in \Sigma_{k-1}$, the action of $\Z/2$ restricts
to a homeomorphism $ S^2 \times S^1 \times \{ \sigma \} \to S^2 \times S^1 \times \{ \overline{\sigma} \}$.
We record these homotopy types for future reference in the following proposition.

\begin{proposition}
Let $k \geq 2$. Then 
\[ \Conf_k^{\ab}(SU_2) \simeq \left\{ \begin{array}{ll}
                                      SU_2 & \text{ if } k = 2, \\
                                           &                    \\ 
                                      \mathop{\coprod} \limits_{i=1}^{(k-1)!/2} S^2 \times S^1 & \text{ if } k \geq 3. \end{array} \right. \]
\end{proposition}

\section{Homological stability}
\label{SecStability}

In this section, we determine stability phenomena for the spaces of configurations of commuting elements
by varying the group over some of the classical families and by varying the number of commuting elements.

\subsection{Stability for classical families}

Recall that $\P$ denotes the family of finite products of groups of the form $U_n$, $SU_n$ and $Sp_n$.

\begin{proposition}
\label{DeHomAConf}
Let $G$ be a group in the family $\P$ of rank $r \geq 2$ and let $F$ be a field
whose characteristic does not divide the order of the Weyl group. If $n \leq r-2$, the
inclusion of $\Conf_k^{\ab}(G)$ in $\Hom(\Z^k,G)$ induces an isomorphism 
\mbox{$ H_n(\Conf_k^{\ab}(G);F) \to H_n(\Hom(\Z^k,G);F)$}.
\end{proposition}

\begin{proof}
Let $T$ be a maximal torus of $G$. By Theorem 3.2 in \cite{GGG}, the inclusion $i_k \colon \Conf_k(T) \to T^k$ 
is $(r-1)$--connected. Since the dimension of $T$ is at least two, the space $\Conf_k(T)$
is path-connected and we can use the relative Hurewicz theorem (see for instance Theorem 4.37 in \cite{H})
to conclude that $i_k$ induces an isomorphism in homology up to dimension $r-2$, hence in
homology with coefficients in $F$ up to dimension $r-2$. Note that the diagram 
\[ 
\diagram
G/T \times \Conf_k(T) \rto^{\quad \phi} \dto & \Conf_k^{\ab}(G) \dto \\
G/T \times T^k \rto & \Hom(\Z^k,G)
\enddiagram
\]
induced by the conjugation action of $G$ is commutative. By the previous paragraph, the vertical map on 
the left induces an isomorphism on $H_n(-;F)$ for $n \leq r-2$. Since this map is equivariant with respect
to the action of the Weyl group $W_G$, it induces an isomorphism on $H_n(-;F)^{W_G}$ for $n \leq r-2$. The
horizontal maps induce isomorphisms 
\begin{align*}
 & H_n(G/T \times \Conf_k(T);F)^{W_G} \to H_n(\Conf_k^{\ab}(G);F), \\
 & H_n(G/T \times T^k;F)^{W_G} \to H_n(\Hom(\Z^k,G);F),
\end{align*}
for all $n$ by Proposition \ref{CohomologyFormula} and Theorem 3.3 in \cite{B}, together with the second part 
of the proof of Proposition 3.2 in \cite{RS}. Therefore the vertical map on the right induces an isomorphism 
on $H_n(-;F)$ for $n \leq r-2$.
\end{proof}

Let $G_r$ be one of the classical groups $U_r$, $SU_{r+1}$ or $Sp_r$. The standard inclusions $ G_r \to G_{r+1}$ 
induce inclusions $\varphi_r \colon \Conf_k^{\ab}(G_r) \to \Conf_k^{\ab}(G_{r+1})$. Recall that a sequence of 
spaces $\{ X_1 \to X_2 \to \ldots \}$ satisfies strong rational homological stability if for each $i \geq 0$, there
exists a positive integer $N(i)$ such that $H_i(X_n;\Q) \to H_i(X_{n+1};\Q)$ is an isomorphism for all $n \geq N(i)$.

\begin{theorem}
\label{StabilityForRank}
For each $k \leq 1$, the sequence $\{ \Conf_k^{\ab}(G_r),\varphi_r \}_{r \geq 1}$ satisfies strong rational
homological stability. Stability holds in homological degree $n$ once
\[ r \geq \left\{ \begin{array}{ll}
                  n+2 & \text{ if } G_r = Sp_r, \\
                  \max \{ (n+k-1)/2 , n+2 \} & \text{ if } G_r = U_r, \\
                  \max \{ (n+k-3)/2 , n+2 \} & \text{ if } G_r = SU_{r+1}. \end{array} \right. \]
\end{theorem}

\begin{proof}
Strong rational homological stability for the sequence of spaces $\Hom(\Z^k,G_r)$ was first established in \cite{RS}, but we will
use the improved bounds from \cite{KT}. By Theorem 7.4 in \cite{KT}, the analogous map $\Hom(\Z^k,G_r) \to \Hom(\Z^k,G_{r+1})$ 
induces an isomorphism on $n$-dimensional rational homology if $n \leq d_{k,r}$ for a certain integer that depends on whether 
$G_r$ is $Sp_r$, $U_r$ or $SU_{r+1}$. And if $r \geq n+2$, in particular we have $r\geq 2$, so we can use Proposition \ref{DeHomAConf}.
Hence the inclusion of $\Conf^{\ab}_k(G_r)$ in $\Hom(\Z^k,G_r)$ induces an isomorphism on $n$-dimensional cohomology. The desired result follows from the commutativity of
\[
\diagram
\Conf_k^{\ab}(G_r) \dto \rto & \Conf_k^{\ab}(G_{r+1}) \dto \\
\Hom(\Z^k,G_r) \rto & \Hom(\Z^k,G_{r+1}), 
\enddiagram
\]
the description of $d_{k,r}$ and the fact that $n+2 \geq (n-1)/2$ in the case of $Sp_r$.
\end{proof}

\begin{remark}
Let $T_r$ be the standard maximal torus of $G_r$. By Corollary 6.3 in \cite{RS},
the $FI_W$--module $\{ H_n(T_r^k;\Q) \}_{r \geq 1}$ is finitely generated 
(see also Proposition 5.3 and 5.9 in \cite{RS}). The inclusion $\Conf_k(T_r) \to T_r^k$ 
is $W_{G_r}$--equivariant, hence the induced morphisms in rational homology form a map of $FI_W$--modules
\[ \{ H_n(\Conf_k(T_r);\Q) \}_{r \geq 1} \to \{ H_n(T_r^k;\Q) \}_{r \geq 1}. \]
Since the inclusion $\Conf_k(T_r) \to T_r^k$ induces an isomorphism in rational 
homology up to dimension $r-2$, the $FI_W$--module $\{ H_n(\Conf_k(T_r);\Q) \}_{r \geq 1}$ is 
also finitely ge-nerated. By Proposition 6.6 in \cite{RS}, the consistent sequence 
$\{ H_i(G_r/T_r;\Q) \}_{r \geq 1}$ of $W_{G_r}$--modules extends to a finitely generated $FI_W$--module
and therefore the $FI_W$--module $\{ H_m(\Conf_k(T_r) \times G_r/T_r;\Q) \}_{r \geq 1}$
is finitely generated by Proposition 5.2 in \cite{Wi}. Therefore this sequence satisfies 
uniform representation stability by Theorem 4.27 in \cite{Wi}. By Theorem 4.20
in \cite{RS}, the sequence $\{ \Conf_k(T_r) \times_{W_{G_r}} G_r/T_r \}_{r \geq 1}$ satisfies
strong rational homological stability. But $H_m(\Conf_k(T_r) \times_{W_{G_r}} G_r/T_r;\Q)$ is
isomorphic to
\[  H_m(\Conf_k(T_r) \times G_r/T_r;\Q)^{W_{G_r}} \cong H_m(\Conf_k^{\ab}(G_r);\Q), \]
so the sequence $\{ \Conf_k^{\ab}(G_r) \}_{r \geq 1}$ satisfies strong rational homological
stability. This gives another proof of the strong rational homological stability part of 
Theorem \ref{StabilityForRank} in the same spirit of \cite{RS}.
\end{remark}

\subsection{Stability with respect to the number of commuting elements}

Next we study stability with respect to the number of elements in the tuple.
For a general compact Lie group $G$, there is no canonical map from $\Conf_k^{\ab}(G)$
to $\Conf_{k+1}^{\ab}(G)$. But we can use the projection maps to define a
co$FI$--space. Recall that $FI$ is the category with objects $[n]=\{1,\ldots,n\}$
and morphisms given by injective maps.

By analogy with \cite{Ch}, we consider the funtor $FI^{\op} \to \Top$ that takes $[n]$ to $\Conf_n^{\ab}(G)$
and an injective map $f \colon [n] \to [m]$ to the map
\begin{gather*}
\Conf_m^{\ab}(G) \to \Conf_n^{\ab}(G) \\
(g_1,\ldots,g_m) \mapsto (g_{f(1)}, \ldots, g_{f(n)}).
\end{gather*}
We are interested in the associated sequences of $\Sigma_r$--representations 
\[ H^j(\Conf_1^{\ab}(G);\Q) \to H^j(\Conf_2^{\ab}(G);\Q) \to \ldots \]
which are induced by the standard inclusions $[r] \to [r+1]$. Next theorem
shows that these sequences satisfy uniform representation stability in the 
sense of \cite{CF}. 

\begin{theorem}
\label{StabilityWithRespectToTuple}
Let $G$ be a group in $\P$ and $j$ a nonnegative integer. Then the
sequences of $\Sigma_k$--representations $ \{ H^j(\Conf_k^{\ab}(G);\Q) \}_{k \geq 1}$
satisfy uniform representation stability.
\end{theorem}

\begin{proof}
For any nonnegative integer $n$, the $FI$--module $\{ H^n(\Conf_k(T);\Q) \}_{k \geq 1}$ is finitely generated by Theorem 1 in \cite{Ch}
and so is the constant $FI$--module $ \{ H^m(G/T;\Q) \}_{k \geq 1}$ for all $m$. By K\"unneth theorem and Proposition 2.3.6 of \cite{CEF}, 
the $FI$--module $ \{ H^j(G/T \times \Conf_k(T);\Q) \}_{k \geq 1}$ is also finitely generated.

The map $\phi \colon G/T \times \Conf_k(T) \to \Conf_k^{\ab}(G)$ is $\Sigma_k$--equivariant
if we consider the actions of $\Sigma_k$ given by permuting coordinates on $\Conf_k(T)$ and 
$\Conf_k^{\ab}(G)$, and the trivial action on $G/T$. Therefore the induced map in rational
cohomology is $\Sigma_k$--equivariant. By Lemma 7.6 in \cite{RS}, the $FI$--module
\[ \{ H^j(G/T \times \Conf_k(T);\Q)^{W_G} \}_{k \geq 1} \cong \{ H^j(\Conf_k^{\ab}(G);\Q) \}_{k \geq 1} \]
is finitely generated. Therefore it satisfies uniform representation stability by
Theorem 1.13 in \cite{CEF}.
\end{proof}

Let $\UConf_k^{\ab}(G) = \Conf_k^{\ab}(G)/\Sigma_k$ be the space of unordered
configurations of $k$ commuting elements. Since the quotient $\Conf_k^{\ab}(G) \to \UConf_k^{\ab}(G)$
induces an isomorphism $H^j(\UConf_k^{\ab}(G);\Q) \to H^j(\Conf_k^{\ab}(G);\Q)^{\Sigma_k}$, we can
use Theorem \ref{StabilityWithRespectToTuple} above and Proposition 4.21 from \cite{RS} to obtain the following corollary.

\begin{corollary}
Let $G$ be a group in $\P$ and $j$ a nonnegative integer. There \mbox{exists}
a positive integer $N(j)$ such that if $ k \geq N(j)$, then \mbox{$H^j(\UConf_k^{\ab}(G);\Q) \cong H^j(\UConf_{k+1}^{\ab}(G);\Q)$}.
\end{corollary}

Note that the maps $\Conf_k^{\ab}(G) \to \Conf_{k-1}^{\ab}(G)$ that forget the
last coordinate do not induce maps $\UConf_k^{\ab}(G) \to \UConf_{k-1}^{\ab}(G)$
and there is no canonical map between these spaces that may induce these isomorphisms.

\section{Unstable cohomology}
\label{SecUnstable}

In this section we include the computation of the cohomology groups of $\Conf_2^{\ab}(G)$ for any group $G$ in
$\P$ of rank two, the cohomology groups of $\Conf_3^{\ab}(U_2)$
and the one-dimensional rational homology and cohomology of $\Conf_k^{\ab}(G)$
for any group $G$ in $\P$ of rank at least two, all of these over fields whose
characteristics do not divide the order of the Weyl group of the respective group.
We also determine the cohomology rings of $\Conf_2^{\ab}(G)$ and $\UConf_2^{\ab}(G)$
over fields of characteristic different from two, for $G=U_2$ and $G = S^1 \times SU_2$.

\begin{proposition}
\label{FirstHomology}
Let $G$ be a group in the family $\P$ of rank at least two and
$F$ a field whose characteristic does not divide the order of the Weyl group
of $G$. Then
\[ H_1(\Conf_k^{\ab}(G);F) \cong \pi_1(G)^k \otimes_{\Z} F, \]
\[ H^1(\Conf_k^{\ab}(G);F) \cong \Hom(\pi_1(G)^k,F). \]
\end{proposition}

\begin{proof}
Let $T$ be a maximal torus of $G$ and $W$ the corresponding Weyl group.
By Theorem 3.2 in \cite{GGG}, the inclusion $\Conf_k(T) \to T^k$ is
$(\dim(T)-1)$--connected. Hence if the dimension of $T$ at least three, 
it induces an isomorphism on the fundamental group. If the dimension of $T$ 
is two, this inclusion induces an epimorphism on the fundamental group, 
but Section 3 of \cite{Bir} shows that the kernel is the commutator subgroup, 
hence the inclusion induces an isomorphism on the first homology group. To
sum up, if $T$ has dimension at least two, the space $\Conf_k(T)$ is path-connected
and this inclusion induces an isomorphism on $H_1(-;F)$. By Proposition \ref{CohomologyFormula}, 
we have that $H^1(\Conf_k^{\ab}(G);F)$ is isomorphic to
\[ \left[ H_0(G/T;F) \otimes_F H_1(\Conf_k(T);F) \oplus H_1(G/T;F) \otimes_F H_0(\Conf_k(T);F) \right]^W . \]
Since the inclusion $\Conf_k(T) \to T^k$ is $W$--equivariant, it induces an isomorphism $H_1(T^k;F)^W \cong H_1(\Conf_k(T);F)^W$,
hence the expression above is isomorphic to
\[ \left[ H_0(G/T;F) \otimes_F H_1(T^k;F) \oplus H_1(G/T;F) \otimes_F H_0(T^k;F) \right]^W \cong H_1(\Hom(\Z^k,G);F). \]
This last isomorphism holds by Theorem 4.3 in \cite{B} and the second part of the proof of 
Proposition 3.2 in \cite{RS}, since $\Hom(\Z^k,G)$ is path-connected
(see Proposition 2.5 and comment after Definition 2.7 in \cite{AG}). Finally, the
inclusion of $\Hom(\Z^k,G)$ in $G^k$ induces an isomorphism on the fundamental group
by Theorem 1.1 in \cite{GPS} and the desired result in homology follows. The cohomology
statement follows directly from the universal coefficient theorem. 
\end{proof}

These formulas do not necessarily hold when the group $G$ has rank one. For example, we saw in
Section \ref{SecConfigurations} that $\Conf_k(S^1) \simeq S^1 \times \Sigma_{k-1}$. Hence
\[ H_1(\Conf_k(S^1);\Q) \cong \Q^{(k-1)!} \ncong \Q^k \cong H_1( (S^1)^k ; \Q) \]
if $k \geq 2$. Note that the fundamental group of these two spaces based at any point do not agree, since the path-components
of $\Conf_k(S^1)$ are contractible. Hence we can not hope to have a result like Theorem 1.1 in \cite{GPS} without
further assumptions.

\begin{remark}
By the commutativity of the diagram
\[ 
\diagram
G/T \times \Conf_k(T) \rto \dto & \Conf_k^{\ab}(G) \dto \\
G/T \times T^k \rto & \Hom(\Z^k,G),
\enddiagram
\]
the isomorphisms from Proposition \ref{FirstHomology} are induced by the inclusion $\Conf_k^{\ab}(G) \to G^k$.
\end{remark}

We now proceed to compute the cohomology groups of $\Conf_2^{\ab}(G)$ for
any group $G$ in $\P$ of rank $2$ over fields whose characteristics do not
divide the order of the Weyl group of $G$. The case of $ T = S^1 \times S^1$ 
corresponds to $\Conf_2(T)$ since $T$ is abelian. By Example 2.6 in \cite{Co},
since $T$ is a topological group, we have
\[ \Conf_2(T) \cong T \times (T-\{1\}) \simeq T \times (S^1 \vee S^1) \]
and its cohomology ring over any field $F$ is easily determined by K\"unneth theorem, namely
\[ H^*(\Conf_2(T);F) \cong \Lambda_F(x_1,y_1) \otimes_F F[z_1,w_1]/(z_1^2,w_1^2,z_1w_1). \]

If $F$ is a field whose characteristic does not divide the order of the Weyl group $W_G$ of $G$, then we have 
$H^*(\Conf_2^{\ab}(G);F) \cong H^*(G/T \times \Conf_2(T);F)^{W_G}$ by Proposition \ref{CohomologyFormula}, so 
we need to determine $H^n(\Conf_2(T);F)$ and $H^n(G/T;F)$ as $FW_G$--modules in each case. It is clear that
$H^0(\Conf_2(T);F)$ and $H^0(G/T;F)$ are the trivial one-dimensional representation, and that $H^n(\Conf_2(T);F)=0$
for $n\geq 4$. Note that $H^1(T;\Z)$ and $H^2(G/T;\Z)$ correspond to the integral representation that defines $W_G$ 
as a finite reflection group over $\Z$. For $U_2$, $SU_3$ and $Sp_2$, this follows from Lemma III.4.16, Corollary III.4.17, 
Theorem III.5.5, Theorem III.5.6 and Corollary III.5.7 of \cite{MT}. For $ S^1 \times SU_2$, it is clear that its Weyl group $\Z/2$ acts
trivially on the first factors of $T$ and $(S^1 \times SU_2)/(S^1 \times S^1) \cong S^1 \times (SU_2/S^1)$, and it acts on
the second factors in the same way as the Weyl group of $SU_2$. We also have $H^1(T;F) \cong H^1(T;\Z) \otimes F$ and $H^2(G/T;F) \cong H^2(G/T;\Z) \otimes F$
in all these cases. We can determine the rest of $H^n(G/T;F)$ as $FW_G$--modules using the product structure
of $H^*(G/T;F)$, which can be found in Theorem III.5.5, Theorem III.5.6 and Corollary III.5.7 in \cite{MT}. Finally,
the inclusion $S^1 \vee S^1 \to T-\{1\}$ is a $W_G$--homotopy equivalence, hence we also know $H^1(S^1 \vee S^1;F)$
as an $FW_G$--module. Using K\"unneth theorem, the character tables of $W_G$ for each $G$, and the previous considerations, 
we find the desired representations which we display in the following table:

\begin{table}[h]
\begin{center}
\begin{minipage}{20cm}
\caption{$FW_G$--module structure of $H^n(\Conf_2(T);F)$ and $H^n(G/T;F)$.}
\label{table1}
\begin{tabular}{|c||c|c||c|c||c|c||c|c|}
\hline 
$n\backslash G$ & \multicolumn{2}{|c||}{$U_2$} & \multicolumn{2}{|c||}{$S^1 \times SU_2$} & \multicolumn{2}{|c||}{$SU_3$} & \multicolumn{2}{|c|}{$Sp_2$}  \\
\hline
 & $\Conf_2(T)$ & $G/T$ & $\Conf_2(T)$ & $G/T$ & $\Conf_2(T)$ & $G/T$ & $\Conf_2(T)$ & $G/T$ \\
 \hline
$0$ & $1$ & $1$ & $1$ & $1$ & $1$ &  $1$ & $1$ & $1$ \\
\hline
$1$ & $2 \oplus 2\sigma$ & $0$ & $2 \oplus 2\sigma$ & $1$ & $2\std$ &  $0$ & $2d$ & $0$ \\
\hline
$2$ & $2 \oplus 3\sigma$ & $\sigma$ & $2 \oplus 3\sigma$ & $\sigma$ & $1 \oplus \std \oplus 2\sgn$ &  $\std$ & $a\oplus b \oplus 2c \oplus 1$ & $d$ \\
\hline
$3$ & $1\oplus \sigma$  & $0$  & $1\oplus \sigma$  & $\sigma$ &  $\std$  & $0$  & $d$ & $0$ \\
\hline
$4$ & $0$  & $0$  & $0$  & $0$ &  $0$  & $\std$  & $0$  & $a \oplus b$ \\
\hline
$5$ & $0$  & $0$  & $0$  & $0$ &  $0$  & $0$  & $0$  & $0$ \\
\hline
$6$ & $0$  & $0$  & $0$  & $0$ &  $0$  & $\sgn$  & $0$  & $d$ \\
\hline 
$7$ & $0$  & $0$  & $0$  & $0$ &  $0$  & $0$  & $0$  & $0$ \\
\hline
$8$ & $0$  & $0$  & $0$  & $0$ &  $0$  & $0$  & $0$  & $c$ \\
\hline
$n \geq 9$ & $0$  & $0$  & $0$  & $0$ &  $0$  & $0$  & $0$  & $0$ \\
\hline
\end{tabular}
\end{minipage}
\end{center}
\end{table}

In this table we are denoting by $1$ the trivial one-dimensional representation of $\Z/2$, $\Sigma_3$ and $D_8$.
The sign representation of $\Z/2$ is denoted by $\sigma$. We denote by $\std$ and $\sgn$ the standard representation 
and the alternating representation of $\Sigma_3$, respectively. The two-dimensional representation of $D_8 = \langle r, s \mid r^4, s^2, (sr)^2 \rangle$ is
denoted by $d$, while $a$, $b$, $c$ are the one-dimensional representations with kernels $\langle s \rangle$, $\langle sr \rangle$ and $\langle r \rangle$,
respectively. Since all these representations are self-dual, the only double tensor products of these representations that contain $1$ as a subrepresentation are 
of the form $V \otimes V$ and therefore we obtain the following expressions in terms of multiplicities of irreducible subrepresentations of $ A^k = H^k(\Conf_2(T);F)$ 
\begin{align*}
 & \dim_F H^n(\Conf_2^{\ab}(U_2);F) = \mult_1 A^n + \mult_{\sigma} A^{n-2}, \\
 & \dim_F H^n(\Conf_2^{\ab}(S^1 \times SU_2);F) = \mult_1 (A^n \oplus A^{n-1}) + \mult_{\sigma} (A^{n-2} \oplus A^{n-3}), \\
 & \dim_F H^n(\Conf_2^{\ab}(SU_3);F) = \mult_1 A^n + \mult_{\std} (A^{n-2} \oplus A^{n-4}) + \mult_{\sgn} A^{n-6}, \\
 & \dim_F H^n(\Conf_2^{\ab}(Sp_2);F)  = \mult_1 A^n + \mult_d (A^{n-2} \oplus A^{n-6}) + \mult_a A^{n-4}  \\
 & \hspace{12em} \, + \mult_b A^{n-4} + \mult_c A^{n-8}.
\end{align*}
where $\mult_V W$ denotes the multiplicity of $V$ as a subrepresentation of $W$. Next proposition is obtained using these equations 
and Table \ref{table1}.

\begin{proposition}
\label{CohomologyGroups}
Let $G$ be a group in $\P$ of rank $2$ and $F$ a field whose characteristic does not divide the order of the Weyl group
of $G$. Then the dimension of $H^n(\Conf_2^{\ab}(G);F)$ as a vector space over $F$ is given by the following table.
\begin{table}[h]
\begin{minipage}{12cm}
\begin{center}
\captionsetup{justification=centering}
\caption{Dimension of $H^n(\Conf_2^{\ab}(G);F)$.}
\label{table2}
\begin{tabular}{|c|c|c|c|c|c|}
\hline 
$n\backslash G$ & $S^1 \times S^1$ & $U_2$ & $S^1 \times SU_2$ & $SU_3$ & $Sp_2$  \\
\hline
$0$ & $1$ & $1$ & $1$ & $1$ & $1$ \\
\hline
$1$ & $4$ & $2$ & $3$ & $0$ & $0$ \\
\hline
$2$ & $5$ & $2$ & $4$ & $1$ & $1$ \\
\hline
$3$ & $2$  & $3$  & $5$  & $2$ &  $2$  \\
\hline
$4$ & $0$  & $3$  & $6$  & $1$ &  $0$  \\
\hline
$5$ & $0$  & $1$  & $4$  & $3$ &  $1$  \\
\hline
$6$ & $0$  & $0$  & $1$  & $1$ &  $2$  \\
\hline 
$7$ & $0$  & $0$  & $0$  & $1$ &  $2$  \\
\hline
$8$ & $0$  & $0$  & $0$  & $2$ &  $0$  \\
\hline
$9$ & $0$  & $0$  & $0$  & $0$ &  $1$  \\
\hline
$10$ & $0$  & $0$  & $0$  & $0$ &  $2$  \\
\hline
$n \geq 11$ & $0$  & $0$  & $0$  & $0$ & $0$ \\
\hline
\end{tabular}
\end{center}
\end{minipage}
\end{table}

\end{proposition}

We could extract from this process multiplicative generators for the cohomology rings of these spaces, although this
is quite tedious for $SU_3$ and $Sp_2$. Instead, we determine the cohomology rings for $\Conf_2^{\ab}(U_2)$ and 
$\Conf_2^{\ab}(S^1 \times SU_2)$ below.

\begin{proposition}
\label{CohomologyRings}
Let $F$ be a field of characteristic different from two. Then 
\begin{align*}
 & H^*(\Conf_2^{\ab}(U_2);F) \cong F[b_1,c_1,d_2,e_3,f_3]/I. \\
 & H^*(\Conf_2^{\ab}(S^1 \times SU_2);F) \cong F[a_1,x_1,z_1,c_2,d_3,e_3]/J,
\end{align*}
where $I$ is the ideal generated by $d_2^2$, $c_1d_2$, $c_1f_3$, $d_2e_3$, $d_2f_3$, $e_3f_3$ and the obvious relations imposed by graded commutativity, 
and $J$ is the ideal generated by $c_2^2$, $z_1c_2$, $z_1e_3$, $c_2d_3$, $c_2e_3$, $d_3e_3$
and the obvious relations imposed by graded commutativity.
\end{proposition}

\begin{proof}
We use the multiplicative generators from
\[ H^*(\Conf_2(T);F) \cong \Lambda_F(x_1,y_1) \otimes_F F[z_1,w_1]/(z_1^2,w_1^2,z_1w_1). \]
Let $a_2$ be a generator of $H^2(U_2/T;F)$. The multiplicative generators of $H^*(\Conf_2^{\ab}(U_2);F)$ coming from the proof of Proposition \ref{CohomologyGroups} are 
\begin{align*}
 & b_1=x_1+y_1, \qquad &  f_3=a_2(z_1-w_1), \\
 & c_1=z_1+w_1, & g_3=x_1y_1(z_1-w_1), \\
 & d_2=(x_1-y_1)(z_1-w_1), & h_4=a_2x_1y_1, \\ 
 & e_3=a_2(x_1-y_1). & 
\end{align*}
Since $b_1d_2 = -2g_3$ and $b_1e_3 = -2h_4$, and the characteristic of $F$ is not two, we can leave $g_3$ and $h_4$ out. All these generators square to zero, hence 
a straightforward computation using Table \ref{table2} shows the desired result.
                         
Let $a_1$ and $b_2$ be generators of $H^1(S^1;F)$ and $H^2(S^2;F)$, respectively. In this case the multiplicative generators of $H^*(\Conf_2^{\ab}(S^1 \times SU_2);F)$ 
coming from the computation of the cohomology groups are $a_1$, $x_1$, $z_1$, $c_2 = y_1w_1$, $ d_3 = y_1 b_2$ and $ e_3 = w_1 b_2$. 
These generators square to zero and the calculation follows similarly.
\end{proof}

We can also extract the cohomology rings of the spaces of unordered configurations of two commuting elements in these
two cases.

\begin{proposition}
\label{CohomologyRingsUnordered}
Let $F$ be a field of characteristic different from two. Then 
\begin{align*}
 & H^*(\UConf_2^{\ab}(U_2);F) \cong \Lambda_F(r_1, s_3), \\
 & H^*(\UConf_2^{\ab}(S^1 \times SU_2);F) \cong \Lambda_F(a_1,u_1, v_3). 
\end{align*}
\end{proposition}

\begin{proof}
Since the characteristic of $F$ is different from two, we have 
\[H^*(\UConf_2^{\ab}(G);F) \cong H^*(\Conf_2^{\ab}(G);F)^{\Z/2} \]
for the action of $\Z/2$ induced by permuting the coordinates of $\Conf_2^{\ab}(G)$. In turn, under the isomorphism $H^*(\Conf_2^{\ab}(G);F) \cong H^*(G/T \times \Conf_2(T);F)^{W_G}$,
the action of $\Z/2$ on the right is induced by permuting the coordinates of $\Conf_2(T)$ and the trivial action on $G/T$. Recall
that the homeomorphism $ \Conf_2(T) \cong T \times (T-\{1\})$ takes $(a,b)$ to $(a,ba^{-1})$, so the corresponding action of $\Z/2$
on $T \times (T-\{1\})$ takes $(x,y)$ to $(yx,y^{-1})$. It is easy to check that the induced action on the cohomology of 
$\Conf_2(T)$ with coefficients in $F$ in terms of the generators used in the proof of Proposition \ref{CohomologyRings} changes 
the sign of $z_1$ and $w_1$, takes $x_1$ to $x_1+z_1$ and $y_1$ to $y_1+w_1$.

Let $\alpha$ be the nontrivial element of $\Z/2$. Its action on the multiplicative generators of $H^*(\Conf_2^{\ab}(U_2);F)$ 
from Proposition \ref{CohomologyGroups} is given by
\begin{align*}
 & \alpha b_1 = b_1+c_1, \qquad & \alpha c_1 = -c_1, \\
 & \alpha e_3 = e_3+f_3, & \alpha f_3 = -f_3, \\
 & \alpha d_2 = -d_2. &
\end{align*}
Hence $H^*(\UConf_2^{\ab}(U_2);F)$ is multiplicatively generated by $r_1 = 2b_1+c_1$ and \mbox{$s_3=2e_3+f_3$}. These elements
square to zero by graded commutativity and their product generates $H^4(\UConf_2^{\ab}(U_2);F)$.

In the case of $S^1 \times SU_2$, again using the multiplicative generators from Proposition \ref{CohomologyGroups}, this action satisfies
\begin{align*}
 & \alpha a_1 = a_1, \qquad & \alpha c_2 = -c_2, \\
 & \alpha x_1 = x_1+z_1, & \alpha z_1 = -z_1, \\
 & \alpha d_3 = d_3+e_3, & \alpha e_3 = -e_3.
\end{align*}
Hence $H^*(\UConf_2^{\ab}(S^1 \times SU_2);F)$ is multiplicatively generated by $a_1$, $u_1=2x_1+z_1$ and $v_3=2d_3+e_3$. These elements
square to zero by graded commutativity and their products generate the second, fourth and fifth cohomology groups of $\UConf_2^{\ab}(U_2)$
with coefficients in $F$.
\end{proof}

We finish the paper with the calculation of the cohomology groups of $\Conf_3^{\ab}(U_2)$ over any field $F$ of
characteristic different from two. We will need an auxiliary result in group cohomology. The notation $F_2$ stands
for the free group on two generators.

\begin{lemma}
\label{CohomologyOfFreeGroup}
Let $\Z/2$ act on $F_2$ by the automorphism which permutes its two gene-rators $a$ and $b$, and let $M$ be an $RF_2$--module
where $\Z/2$ acts $R$--linearly and satisfying $x(zm)=(xz)(xm)$ for $x\in \Z/2$, $z \in F_2$ and $m\in M$. Then 
there is an isomorphism of $R\Z/2$--modules
\[ H^1(F_2;M) \cong (M \oplus M)/N, \]
where $N$ is the $R$--submodule generated by the elements of the form $(am-m,bm-m)$ with $m \in M$,
and the action of the nontrivial element $\alpha$ of $\Z/2$ on $M \oplus M$ is given by 
\[ \alpha(m,n) = (\alpha n, \alpha m). \]
\end{lemma}

\begin{proof}
The expression
\[ H^1(F_2;M) \cong \Der(F_2,M)/P(F_2,M) \]
as a quotient of the group of derivations by the subgroup of principal
derivations is natural in $F_2$ and $M$, and any derivation
of $F_2$ is uniquely determined by its values on $a$ and $b$ (see for instance
Exercise IV.2.3 in \cite{Br}). Therefore
$\Der(F_2,M) \cong M \oplus M$ and $P(F_2,M)=N$ with the desired action
of $\Z/2$.
\end{proof}

Since $H^*(\Conf_3^{\ab}(U_2);F) \cong H^*(U_2/T \times \Conf_3(T);F)^{\Z/2}$, we compute first $H^*(\Conf_3(T);F)$ as an $F\Z/2$--module. 
As before, we denote by $1$ the trivial one-dimensional representation and by $\sigma$ the sign representation of $\Z/2$. 

\begin{lemma}
\label{CohomologyOfConfiguration}
Let $\Z/2$ act on $T=S^1 \times S^1$ by permuting coordinates, and let $F$ be a field of characteristic different from two. Then
we have
\[ H^n(\Conf_2(T-\{1\});F) \cong \left\{ \begin{array}{ll}
                                         1 & \text{ if } n=0, \\
                                         2 \oplus 2\sigma & \text{ if } n=1, \\
                                         3 \oplus 2\sigma & \text{ if } n=2, \\
                                         0 & \text{ if } n \geq 3, \end{array} \right. \]

\[ H^n(\Conf_3(T);F) \cong \left\{ \begin{array}{ll}
                                         1 & \text{ if } n=0, \\
                                         3 \oplus 3\sigma & \text{ if } n=1, \\
                                         7 \oplus 7\sigma & \text{ if } n=2, 3, \\
                                         2 \oplus 3\sigma & \text{ if } n=4, \\
                                         0 & \text{ if } n \geq 5, \end{array} \right. \]
as $F\Z/2$--modules.
\end{lemma}

\begin{proof}
Let $x_j=(e^{2\pi j/8},e^{2\pi j/8})$ for $j=1,2,3$. We take $(x_1,x_2,x_3)$ as the basepoint of $\Conf_3(T)$. 
By Theorem 5 in \cite{Bir}, the fundamental group of $\Conf_3(T)$ has a presentation with generators 
$a_1$, $a_2$, $a_3$, $b_1$, $b_2$, $b_3$ and $B_{23}$ and relations
\begin{align*}
 & [a_i,a_j]=[b_i,b_j]=1, \\
 & [a_i,B_{23}]=[b_i,B_{23}] = 1, \\
 & [b_3,a_3a_2^{-1}] = [b_3b_2^{-1},a_3] = B_{23}, \\
 & [a_1,b_i]=[b_1,a_i]=1.
\end{align*}
Let $\gamma_1(t)=(e^{-2\pi i t},1)$ and $\gamma_2(t)=(1,e^{-2\pi i t})$. Explicit representatives for these elements are given by the loops 
\begin{align*}
 & a_1(t)=(x_1\gamma_1(t),x_2\gamma_1(t),x_3\gamma_1(t)), \\
 & a_2(t)=(x_1,x_2\gamma_1(t),x_3\gamma_1(t)), \\
 & a_3(t)=(x_1,x_2,x_3\gamma_1(t)), \\
 & b_j(t) = \alpha a_j(t), \\
 & B_{23}(t) = (x_1,x_2,x_3\eta(t)), 
\end{align*}
where $\alpha$ is the nontrivial element of $\Z/2$ and $\eta$ is the loop based at $1$ corresponding to $-\gamma_{2,3}$ in page 7 of \cite{Bir}. 
If we regard the torus as the usual quotient of a square, the loop $x_3\eta$ corresponds to Figure \ref{figure1} below. Note that $\alpha$ permutes $a_j$ with $b_j$ and changes the sign of $B_{23}$.
\begin{figure}[h!]
\centering
\captionsetup{justification=centering}
\includegraphics{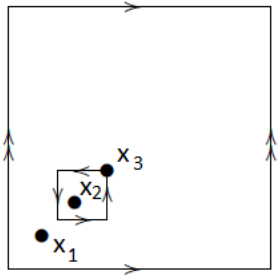}
\caption{The loop $x_3\eta$.}
\label{figure1}
\end{figure}

By Example 2.6 in \cite{Co}, there is a homeomorphism between $\Conf_3(T)$ and \mbox{$T \times \Conf_2(T-\{1\})$} which sends $(y_1,y_2,y_3)$ to $(y_1,y_2y_1^{-1},y_3y_1^{-1})$,
which is clearly $\Z/2$--equivariant. We use this homeomorphism  together with the projection \mbox{$T \times \Conf_2(T-\{1\}) \to \Conf_2(T-\{1\})$} to obtain a presentation of the fundamental 
group of $\Conf_2(T-\{1\})$ based at $(x_1,x_2x_1^{-1},x_3x_1^{-1})$. This presentation has generators $\widetilde{B}_{23}$,
$\widetilde{a}_j$ and $\widetilde{b}_j$ for $j=2,3$, which are the images of the corresponding loops on $\Conf_3(T)$ under
the composition of those two maps. They satisfy the relations
\begin{align*}
 & [\widetilde{a}_i,\widetilde{a}_j]=[\widetilde{b}_i,\widetilde{b}_j]=1, \\
 & [\widetilde{a}_i,\widetilde{B}_{23}]=[\widetilde{b}_i,\widetilde{B}_{23}] = 1, \\
 & [\widetilde{b}_3,\widetilde{a}_3\widetilde{a}_2^{-1}] = [\widetilde{b}_3\widetilde{b}_2^{-1},\widetilde{a}_3] = \widetilde{B}_{23}.
\end{align*}
The action of $\alpha$ permutes $\widetilde{a}_j$ with $\widetilde{b}_j$ and changes the sign of $\widetilde{B}_{23}$, hence
\[ H^1(\Conf_2(T-\{1\};F) \cong \Hom(\pi_1(\Conf_2(T-\{1\}),F) \cong 2\oplus 2\sigma. \]

The fiber of the Fadell-Neuwirth fibration $\Conf_2(T-\{1\}) \to T-\{1\}$ over $x_1$ is $T-\{1,x_1\}$ (see Theorem 1 in \cite{FN}).
Since $T-\{1\} \simeq S^1 \vee S^1$ and $T-\{1,x_1\} \simeq S^1 \vee S^1 \vee S^1$, we obtain that
$H^n(\Conf_2(T-\{1\});F)=0$ if $n\geq 3$. The associated Serre spectral sequence in cohomology 
with coefficients in $F$ collapses at $E_2$ and
\[ H^2(\Conf_2(T-\{1\});F) \cong H^1(T-\{1\};\underline{H^1(T-\{1,x_1\};F)}), \]
where the underlined term indicates local coefficients. We first determine the action of $\pi_1(T-\{1\},x_1)$
on $\pi_1(T-\{1,x_1\},x_2)$ using the extension of groups
\[ 1 \to \pi_1(T-\{1,x_1\},x_2) \stackrel{i_*}{\longrightarrow} \pi_1(\Conf_2(T-\{1\}),(x_1,x_2)) \stackrel{p_*}{\longrightarrow} \pi_1(T-\{1\},x_1) \to 1. \]
The group $\pi_1(T-\{1\},x_1) \cong F_2$ is generated as a free group by the classes $h$ and $v$ of the 
loops $x_1\gamma_1(t)$ and $x_1\gamma_2(t)$. The group $\pi_1(T-\{1,x_1\},x_2) \cong F_3$ is
generated as a free group by the classes $\tilde{h}$, $\tilde{v}$ and $\tilde{r}$ of the loops
$x_2\gamma_1(t)$, $x_2\gamma_2(t)$ and $x_3\eta(t)$. It is easy to check that
\begin{alignat*}{2}
& i_*(\tilde{h}) = \tilde{a}_3, \qquad \qquad && p_*(\tilde{a}_2) = h, \\
& i_*(\tilde{v}) = \tilde{b}_3, && p_*(\tilde{b}_2) = v, \\
& i_*(\tilde{r}) = \tilde{B}_{23}, && p_*(\tilde{a}_3) = p_*(\tilde{b}_3) = 1, 
\end{alignat*}  
and therefore $h$ and $v$ act via conjugation by $\tilde{a}_2$ and $\tilde{b}_2$, respectively.
A straightforward computation using the presentation of $\pi_1(T-\{1,x_1\},x_2)$ shows that the actions
of $h$ and $v$ are given by
\begin{alignat*}{2}
& h \cdot \tilde{h} = \tilde{h}, \qquad \qquad \qquad && v \cdot \tilde{h} = \tilde{r}^{-1}\tilde{v}\tilde{h}\tilde{v}^{-1}, \\
& h \cdot \tilde{v} = \tilde{r}\tilde{h}\tilde{v}\tilde{h}^{-1}, && v \cdot \tilde{v} = \tilde{v}, \\
& h \cdot \tilde{r} = \tilde{r}, && v \cdot \tilde{r} = \tilde{r}.
\end{alignat*}
Let us consider the basis $\{ e_h, e_v, e_r \}$ of $H^1(T-\{1,x_1\};F)$ corresponding to the
elements $\tilde{h}$, $\tilde{v}$ and $\tilde{r}$. Then the action of $h$ and $v$ on $H^1(T-\{1,x_1\};F)$
is given by the matrices
\[ \left( \begin{array}{ccc}
           1 & 0 & 0 \\
           0 & 1 & 1 \\
           0 & 0 & 1 \end{array} \right),   \qquad \qquad \left( \begin{array}{ccc}
           1 & 0 & -1 \\
           0 & 0 & 0 \\
           0 & 1 & 1 \end{array} \right). \]
Since $T-\{1\}$ is a classifying space for $F_2$, we can use Lemma \ref{CohomologyOfFreeGroup} to obtain
\[ H^1(T-\{1\};\underline{H^1(T-\{1,x_1\};F)}) \cong \frac{F\{e_h,e_v,e_r\} \oplus F\{e'_h,e'_v,e'_r\}}{F\{(he_j-e_j,ve'_j-e'_j) \mid j=h,v,r\}}. \]
Note that the quotient identifies $e_v$ with $e_h'$, hence this quotient is generated by
the classes of $e_h$, $e_v$, $e_r$, $e_v'$ and $e_r'$ as an $F$--module. The action of $\Z/2$ sends
\[ e_h \mapsto e_v', \qquad e_v \mapsto e_h' = e_v, \qquad e_r \mapsto -e_r', \]
hence $H^2(\Conf_2(T-\{1\});F) \cong 3 \oplus 2\sigma $ as an $F\Z/2$--module.
Since $H^1(T;F)$ equals $1 \oplus \sigma$ and $H^2(T;F)=\sigma$, the statement about 
the cohomology of $\Conf_3(T)$ follows mechanically from K\"unneth theorem.
\end{proof}

\begin{remark}
The computation of $H^1(T-\{1\};\underline{H^1(T-\{1,y\};F)})$ as an $F\Z/2$--module was performed
in \cite{Ji} using the Mayer-Vietoris sequence.
\end{remark}

\begin{proposition}
If $F$ is a field of characteristic different from two, then
\[ H^n(\Conf_3^{\ab}(U_2);F) \cong \left\{ \begin{array}{ll}
                                           F & \text{ if } n=0, \\
										   F^3 & \text{ if } n=1, 6, \\
										   F^7 & \text{ if } n=2, 5, \\
										   F^{10} & \text{ if } n=3, \\
										   F^9 & \text{ if } n=4, \\
										   0 & \text{ otherwise.} \end{array} \right. \]
\end{proposition}

\begin{proof}
Since $H^*(\Conf_3^{\ab}(U_2);F) \cong H^*(U_2/T \times \Conf_3(T);F)^{\Z/2}$ and the action of $\Z/2$ on $U_2/T \cong S^2$ 
corresponds to the antipodal action, just as in Proposition \ref{CohomologyGroups} we have
\[ \dim_F H^n(\Conf_3^{\ab}(U_2);F) = \mult_1 H_n(\Conf_3(T);F) + \mult_{\sigma} H_{n-2}(\Conf_3(T);F). \]
The result follows from Lemma \ref{CohomologyOfConfiguration}.
\end{proof}

\subsection*{Acknowledgments} 
The authors thank Omar Antol\'in for relevant comments which helped
achieve some of the computations.

\subsection*{Funding}
Both authors were partially supported by SEP-CONACYT Basic Science Project 242186: Homotopical aspects of compact Lie groups,
granted by the Mexican institutions Secretar\'ia de Educaci\'on P\'ublica and Consejo Nacional de Ciencia y Tecnolog\'ia.

\bibliographystyle{amsplain}

\bibliography{mybibfile}

\end{document}